\documentclass[12pt]{amsart}

\usepackage{amsfonts, amsthm, amsmath}

\usepackage{rotating}
\usepackage{tikz}

\usepackage{amscd}

\usepackage[latin2]{inputenc}

\usepackage{t1enc}
\usetikzlibrary{shapes,arrows}
\tikzstyle{pnt}=[draw,circle,fill,inner sep=1pt]
\tikzstyle{opnt}=[draw,circle,inner sep=2pt]

\usepackage[mathscr]{eucal}

\usepackage{indentfirst}

\usepackage{graphicx}

\usepackage{graphics}

\usepackage{pict2e}

\usepackage{mathrsfs}

\usepackage{enumerate}
\usepackage[pagebackref]{hyperref}
\hypersetup{backref, pagebackref, colorlinks=true}
\usepackage{cite}
\usepackage{color}
\usepackage{epic}
\usepackage{hyperref} 
\numberwithin{equation}{section}
\theoremstyle{plain}

\newtheorem{theorem}{Theorem}
\newtheorem{definition}[theorem]{Definition}
\newtheorem{lemma}[theorem]{Lemma}

\newtheorem{cor}[theorem]{Corollary}
\newtheorem{remark}[theorem]{Remark}

\newcommand{\Z}{\mathbb{Z}}
\def\DT{\mathrm{DT}}

\def\S{\mathfrak{S}}
\def\L{\mathfrak{L}}
\def\D{\mathfrak{D}}
\def\pp{\operatorname{(2--13)}}
\def\Dc{\mathcal{DC}}
\def\Ec{\mathcal{EC}}

\makeatother

\usepackage{xcolor}

\allowdisplaybreaks

\begin{document}

\title[Cycles on a multiset with only even-odd drops]
{Cycles on a multiset with only even-odd drops}

\author[Z. Lin]{Zhicong Lin}
\address[Zhicong Lin]{Research Center for Mathematics and Interdisciplinary Sciences, Shandong University, Qingdao 266237, P.R. China}
\email{linz@sdu.edu.cn}

\author[S.H.F. Yan]{Sherry H.F. Yan}
\address[Sherry H.F.  Yan]{Department of Mathematics,
Zhejiang Normal University, Jinhua 321004, P.R. China}
\email{hfy@zjnu.cn}

\date{\today}

\begin{abstract}
For a finite subset $A$ of $\Z_{>0}$, Lazar and Wachs (2019) conjectured that the number of  cycles on $A$ with only even-odd drops  is equal to the number of D-cycles on $A$. In this note, we introduce cycles on a multiset with only even-odd drops  and prove bijectively a multiset version of their conjecture. As a consequence, the number of cycles on $[2n]$ with only even-odd drops equals the Genocchi number $g_n$. With Laguerre histories as an intermediate structure, we also construct a bijection between a class of  permutations of length $2n-1$ known to be counted by $g_n$ invented by Dumont and  the cycles on $[2n]$ with only even-odd drops.  
\end{abstract}



\maketitle


  \section{Introduction}
 The {\em Genocchi numbers} 
 $$\{g_n\}_{n\geq1}=\{1,1,3,17,155,2073,38227, 929569,\ldots\}
 $$
 can be defined by the following exponential generating function formula
 $$
 \sum_{n\geq1}g_n\frac{x^{2n}}{(2n)!}=x\tan\frac{x}{2}.
 $$
 Let $\S_n$ be the set of permutations of $[n]:=\{1,2,\ldots,n\}$. Dumont~\cite{du} introduced the class of permutations 
 $$
 \D_{2n-1}:=\{\sigma\in\S_{2n-1}:\forall\, i\in[2n-2], \text{ $\sigma(i)>\sigma(i+1)$ if and only if $\sigma(i)$ is even}\}
 $$
 and proved that 
 \begin{equation}\label{eq:dumont}
 |\D_{2n-1}|=g_n.
 \end{equation}
  For instance, $\D_{5}=\{42135,21435,34215\}$ and so $|\D_5|=3=g_3$.
 Based on Dumont's interpretation of Genocchi numbers, the objective of this paper is to present two different bijective proofs of a conjecture due to Lazar and Wachs~\cite[Conjecture~6.4]{Wachs} which asserts that  cycles on $[2n]$ with only even-odd drops are also counted by $g_n$. Actually, we will also introduce cycles on a multiset with only even-odd drops and prove bijectively a multiset generalization of another related conjecture of Lazar and Wachs~\cite[Conjecture~6.5]{Wachs}. 
 
 Let $M$ be a finite multiset with exactly $m$ elements from $\Z_{>0}$. We introduce cycles on $M$ with only even-odd drops and D-cycles on $M$, generalizing the concepts introduced in~\cite{Wachs} for $M$ being a subset of $\Z_{>0}$. 
 \begin{definition}
 A cycle $(a_1,a_2,\ldots,a_m)$ on $M$ (i.e., as multiset $\{a_i:1\leq i\leq m\}=M$) is said to {\bf\em have only even-odd drops} if whenever $a_i>a_{i+1}$ (as a cycle, we usually assume $a_{m+1}=a_1$), then $a_i$ is even and $a_{i+1}$ is odd; it is a {\bf\em D-cycle} if $a_i\leq a_{i+1}$ when $a_i$ is odd and $a_i\geq a_{i+1}$ when $a_i$ is even for all $1\leq i\leq m$. For example, as cycles on the multiset $\{1^2,2^2,3^2,4^2\}$,  $(1,2,1,2,4,3,3,4)$ has only even-odd drops but is not a D-cycle, while   $(1,2,1,4,3,3,4,2)$ is a D-cycle that does not have only even-odd drops. Let $\Ec_M$ be the set of cycles on $M$ with only even-odd drops and let $\Dc_M$ be the set of D-cycles on $M$. 
 
 Note that in order for $\Ec_M$ or $\Dc_M$ to be non empty, the smallest element in $M$ must be odd and the greatest element in $M$ must be even.   When $M=[n]$, we simply write $\Ec_n$ and $\Dc_n$ for $\Ec_{[n]}$ and $\Dc_{[n]}$, respectively. For example, 
 \begin{align*}
 \Ec_6=\{(1,2,3,4,5,6),(1,2,4,3,5,6),(1,2,5,6,3,4)\},\\
 \Dc_6=\{(1,3,5,6,4,2),(1,4,3,5,6,2),(1,5,6,3,4,2)\}.
 \end{align*}
 \end{definition}
 
 The next result proves~\cite[Conjecture~6.4]{Wachs} bijectively\footnote{We learnt that Qiongqiong Pan and Jiang Zeng~\cite{Pan} have also proved~\cite[Conjecture~6.4]{Wachs} independently using continued fractions. } in view of Dumont's result~\eqref{eq:dumont}. 
 
 \begin{theorem}\label{thm:dum}
 There exists two bijections, $\Phi$ and $\Psi$, between $\Ec_{2n}$ and $\D_{2n-1}$. 
 \end{theorem}
 
 The construction of $\Phi$ is based on the classical  {\em Fran\c{c}on--Viennot bijection}~\cite{Fran} that encodes permutations as Laguerre histories, while $\Psi$ is the composition of the  bijection $\psi$ below with a simple transformation. 
 
 \begin{theorem}\label{thm:mul}
 For a fixed multiset $M$, there exists a bijection $\psi$ between $\Ec_M$ and $\Dc_M$. 
 \end{theorem}
 
 For $M$ being a subset of $\Z_{>0}$, the above theorem proves~\cite[Conjecture~6.5]{Wachs} bijectively. Thus, Theroem~\ref{thm:mul} is a multiset generalization of~\cite[Conjecture~6.5]{Wachs}.  As an immediate consequence of Theorems~\ref{thm:dum} and~\ref{thm:mul} and Dumont's result~\eqref{eq:dumont},  we have
 \begin{cor}
For $n\geq1$, $|\Ec_{2n}|=|\Dc_{2n}|=g_n$.
 \end{cor}
 
 \begin{remark}
Let  $\Ec_{2n}^{(k)}:=\Ec_M$ with $M=\{1^k,2^k,\ldots,(2n)^k\}$. 
 Since $|\Ec_{2n}|=g_n$, $|\Ec_{2n}^{(k)}|$ can be considered as a new generalization of the Genocchi numbers. For another generalization of the Genocchi numbers using the model of trees, the reader is referred to~\cite{Han}. Can the generating function for $|\Ec_{2n}^{(k)}|$ be calculated? Is there any divisibility property for $|\Ec_{2n}^{(k)}|$ similar to $g_n$ (see~\cite{Han})? 
 \end{remark}
 
 The rest of this paper is organized as follows. After recalling the Fran\c{c}on--Viennot bijection, we construct $\Phi$ in Section~\ref{sec:2}. In Section~\ref{sec:3}, we first present the bijection $\psi$ for Theorem~\ref{thm:mul} and then use it to construct $\Psi$. Finally, in Section~\ref{sec:4}, we provide an Inclusion-Exclusion approach to Dumont's result~\eqref{eq:dumont} for the sake of completeness.

 \section{The construction of $\Phi$}
 \label{sec:2}
 In order to construct $\Phi$, we need to recall the classical Fran\c{c}on--Viennot bijection~\cite{Fran} first.
 A {\em Motzkin path} of length $n$ is a lattice path in the first  quadrant starting from $(0,0)$, ending at $(n,0)$, and using three possible steps: 
 $$
 U=(1,1) \text{ (up step), } L=(1,0) \text{ (level step) } \text{ and } D=(1,-1) \text{ (down step).}
 $$ 
 A {\em$2$-Motzkin path} is a Motzkin path in which each level step is further distinguished into two different types of level steps $L_0$ or $L_1$. 
 A $2$-Motzkin paths will be represented as a word over the alphabet $\{U,D,L_0,L_1\}$. A {\em Laguerre history} of length $n$ is a pair $(w,\mu)$, where $w=w_1\cdots w_n$ is a $2$-Motzkin path and $\mu=(\mu_1,\cdots,\mu_n)$ is a vector satisfying $0\leq\mu_i\leq h_i(w)$, where 
 $$h_i(w):=\#\{j\mid j<i, w_j=U\}-\# \{j\mid j<i, w_j=D\}$$
  is the {\em height} of the starting point of the $i$-th step of $w$. Denote by $\L_n$ the set of all Laguerre histories of length $n$. 
 
 For a permutation $\sigma\in\S_n$,  a letter $\sigma(i)$ is called a {\em valley} (resp.~{\em peak, double descent, double ascent}) of $\sigma$ if $\sigma(i-1)>\sigma(i)<\sigma(i+1)$ (resp.~$\sigma(i-1)<\sigma(i)>\sigma(i+1)$, $\sigma(i-1)>\sigma(i)>\sigma(i+1)$, $\sigma(i-1)<\sigma(i)<\sigma(i+1)$), where we use the assumption $\sigma(0)=\sigma(n+1)=0$. For a fixed $1\leq k\leq n-1$, define 
 $$
\pp_k\sigma=\#\{i: i-1>j\text{ and } \sigma(i-1)<\sigma(j)=k<\sigma(i)\}.
$$
 The Fran\c{c}on--Viennot bijection  $\phi_{FV}:\S_n\rightarrow\L_{n-1}$ that we need is the following modified version (see for example~\cite{Shin-zeng2012}) defined as
 $\phi_{FV}(\sigma)=(w,\mu)\in\L_{n-1}$, where for each $i\in[n-1]$: 

  $$
  w_i=\left\{
  \begin{array}{ll}
  U& \mbox{if $i$ is a valley of $\sigma$},  \\
 D&\mbox{if $i$ is a peak of $\sigma$},  \\
   L_0&\mbox{if $i$ is a double ascent of $\sigma$},\\
  L_1 &\mbox{if $i$ is a double descent of $\sigma$},
  \end{array}
  \right.
  $$
  and
$\mu_i=\pp_i(\sigma)$. For example, if $\sigma=528713649\in\S_9$, then $\phi_{FV}(\sigma)=(w,\mu)$, where $w=UUL_0UDDL_1D$ and $\mu=(0,1,0,0,3,1,1,1)$. The reverse algorithm $\phi_{FV}^{-1}$ building a permutation $\sigma$ from a Laguerre history $(w,\mu)\in\L_{n-1}$ can be described  iteratively as:
\begin{itemize}
\item Initialization: $\sigma=\circ$;
\item At the $i$-th ($1\leq i\leq n-1$) step of the algorithm, replace the $(h_i(w)+1)$-th $\circ$ (from right to left) of $\sigma$ by
$$
\begin{cases}
\,\circ i\circ& \text{if $w_i=U$},\\
\, i\circ& \text{if $w_i=L_0$},\\
 \, \circ i& \text{if $w_i=L_1$},\\
  \, i& \text{if $w_i=D$};
\end{cases}
$$
\item The final permutation is obtained by replacing  the last remaining $\circ$ by $n$. 
\end{itemize}
For example, if $(w,\mu)=(UL_1UDL_0D,(0,1,1,2,0,0))\in\L_{6}$, then $\sigma$ is built as follows:
$$
\sigma=\circ\rightarrow \circ1\circ\rightarrow \circ21\circ\rightarrow \circ3\circ21\circ\rightarrow 43\circ21\circ\rightarrow 43\circ215\circ\rightarrow 43\circ2156\rightarrow 4372156.
$$

 Now we begin to construct $\Phi$ step by step.  Let $\S_{2n-1}^{oe}$ be the set of permutations $\sigma\in\S_{2n-1}$ with only {\em odd-even descents} (i.e., whenever $\sigma(i)>\sigma(i+1)$, then $\sigma(i)$ is odd and $\sigma(i+1)$ is even) and whose last entry is odd. For instance, $\S_{5}^{oe}=\{12345,13245,14523\}$.
 
 \begin{lemma}\label{lem:5}
 There exists a bijection $\eta: \Ec_{2n}\rightarrow\S_{2n-1}^{oe}$.
 \end{lemma}
\begin{proof}
For a cycle $\alpha=(a_1,a_2,\ldots,a_{2n})\in\Ec_{2n}$ with $a_1=1$, define $\eta(\alpha)$ to be the permutation $a_2-1,a_3-1,\ldots,a_{2n}-1$ (in one line notation), which is clearly in $\S_{2n-1}^{oe}$. It is easy to see that $\eta$ sets up an one-to-one correspondence between $\Ec_{2n}$ and $\S_{2n-1}^{oe}$.
\end{proof}

 Let us consider the subset $M_{2n}$ of Laguerre histories $(w,\mu)\in\L_{2n}$ with the restriction that
 $$
 w_i=\begin{cases}
 D\text{ or } L_0, \quad\text{if $i$ is odd},\\
 U\text{ or } L_0, \quad\text{if $i$ is even}.
 \end{cases}
 $$
 For instance, 
 $$M_4=\{(L_0UDL_0,(0,0,0,0)),(L_0UDL_0,(0,0,1,0)),(L_0L_0L_0L_0,(0,0,0,0))\}.$$
 
 \begin{lemma}\label{lem:6}
 The Fran\c{c}on--Viennot bijection  $\phi_{FV}$ restricts to a bijection between $\S_{2n-1}^{oe}$ and $M_{2n-2}$. 
 \end{lemma}
 
 \begin{proof}
 This follows from the observation that $\sigma\in\S_{2n-1}$ is a permutation in $\S_{2n-1}^{oe}$ if and only if for each $i\in[2n-2]$, the letter $i$ of $\sigma$ is a double ascent or a peak whenever $i$ is odd and is a double ascent or a valley whenever $i$ is even. 
 \end{proof}
 
 Let us consider another subset $M_{2n}^*$ of Laguerre histories $(w,\mu)\in\L_{2n}$ with the restriction that
 $$
 w_i=\begin{cases}
 U\text{ or } L_0, \quad\text{if $i$ is odd},\\
 D\text{ or } L_1, \quad\text{if $i$ is even},
 \end{cases}
 $$
 and $1\leq\mu_i\leq h_i(w)$ when $i$ is even. 
 For instance, 
 $$M_4^*=\{(UDUD,(0,1,0,1)),(UL_1L_0D,(0,1,0,1)),(UL_1L_0D,(0,1,1,1))\}.$$
 
  \begin{lemma}\label{lem:7}
 The Fran\c{c}on--Viennot bijection  $\phi_{FV}$ restricts to a bijection between $\D_{2n-1}$ and $M_{2n-2}^*$. 
 \end{lemma}
 \begin{proof}
 Observe that $\sigma\in\S_{2n-1}$ is a permutation in $\D_{2n-1}$ if and only if 
(i) $\sigma(2n-1)=2n-1$ and (ii) for each $i\in[2n-2]$, the letter $i$ of $\sigma$ is a double ascent or a valley whenever $i$ is odd and is a double descent or a peak whenever $i$ is even. Thus, if $\sigma\in\D_{2n-1}$, then $\phi_{FV}(\sigma)\in M_{2n-2}^*$ (as $\sigma(2n-1)=2n-1$ forces $\mu_i\geq1$ when $i$ is even). Conversely, if $(w,\mu)\in M_{2n-2}^*$, then as $\mu_i\geq1$ when $i$ is even, it follows from the  iterative construction of $\sigma=\phi_{FV}^{-1}(w,\mu)$ that the $\circ$ at the end of $\sigma$ remains until the last step, i.e., $\sigma(2n-1)=2n-1$. Therefore, we have $\phi_{FV}^{-1}(w,\mu)\in\D_{2n-1}$ for any $(w,\mu)\in M_{2n-2}^*$. 
 \end{proof}

 \begin{lemma}\label{lem:8}
 There exists a bijection $\rho:M_{2n}\rightarrow M_{2n}^*$.
 \end{lemma}
  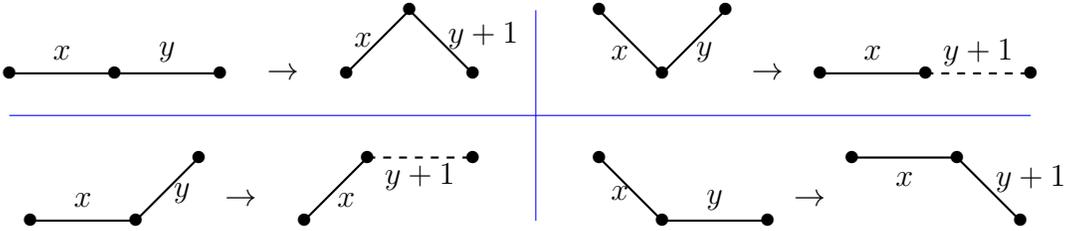
\begin{figure}[h]
\centering
\begin{tikzpicture}[scale=0.28]
\node at (0,0) {$\bullet$};
\node at (5,0) {$\bullet$};\node at (10,0) {$\bullet$};

\node at (16,0) {$\bullet$};\node at (19,3) {$\bullet$};\node at (22,0) {$\bullet$};

\node at (31,0) {$\bullet$};\node at (28,3) {$\bullet$};\node at (34,3) {$\bullet$};

\node at (38.5,0) {$\bullet$};\node at (43.5,0) {$\bullet$};\node at (48.5,0) {$\bullet$};

\node at (1,-7) {$\bullet$};\node at (6,-7) {$\bullet$};\node at (9,-4) {$\bullet$};
\node at (14,-7) {$\bullet$};\node at (17,-4) {$\bullet$};\node at (22,-4) {$\bullet$};

\node at (28,-4) {$\bullet$};\node at (31,-7) {$\bullet$};\node at (36,-7) {$\bullet$};
\node at (40,-4) {$\bullet$};\node at (45,-4) {$\bullet$};\node at (48,-7) {$\bullet$};
\node at (2.5,1) {$x$};\node at (7.5,1) {$y$};
\node at (29,1) {$x$};\node at (33,1) {$y$};
\node at (3.5,-6) {$x$};\node at (8.2,-5.7) {$y$};
\node at (29,-5.7) {$x$};\node at (33.5,-6) {$y$};
\node at (13,0) {$\rightarrow$};\node at (36,0) {$\rightarrow$};\node at (11,-6) {$\rightarrow$};
\node at (38,-6) {$\rightarrow$};
\node at (16.8,1.6) {$x$};\node at (22.5,1.8) {$y+1$};
\node at (41,1) {$x$};\node at (46,1) {$y+1$};
\node at (16,-6) {$x$};\node at (19.5,-4.9) {$y+1$};
\node at (42.5,-5) {$x$};\node at (48.5,-5) {$y+1$};
\draw[-,color=blue] (25,3) to (25,-7);\draw[-,color=blue] (0,-2) to (48.5,-2);
\draw[thick] (0,0) to (5,0);
\draw[thick] (5,0) to (10,0);
\draw[thick] (16,0) to (19,3);\draw[thick] (22,0) to (19,3);
\draw[thick] (31,0) to (28,3);\draw[thick] (31,0) to (34,3);
\draw[thick] (38.5,0) to (43.5,0);\draw[thick,dashed] (43.5,0) to (48.5,0);
\draw[thick] (1,-7) to (6,-7);\draw[thick] (6,-7) to (9,-4);
\draw[thick] (14,-7) to (17,-4);\draw[thick,dashed] (17,-4) to (22,-4);
\draw[thick] (28,-4) to (31,-7);\draw[thick] (31,-7) to (36,-7);
\draw[thick] (40,-4) to (45,-4);\draw[thick] (48,-7) to (45,-4);
\end{tikzpicture}
\caption{The construction of $\rho$: a dashed level step represents $L_1$. \label{fig:rho}}
\end{figure}
 \begin{proof}
 For a Laguerre history $(w,\mu)\in M_{2n}$, we construct $\rho(w,\mu)=(w',\mu')$ by transforming each consecutive two steps $(w_{2i-1},w_{2i})$ and their weights $(\mu_{2i-1},\mu_{2i})=(x,y)$ ($1\leq i\leq n$) according to the following four cases (see Fig.~\ref{fig:rho}):
 \begin{itemize}
 \item If $(w_{2i-1},w_{2i})=(L_0,L_0)$, then $(w_{2i-1}',w_{2i}')=(U,D)$ and $(\mu_{2i-1}',\mu_{2i}')=(x,y+1)$;
  \item If $(w_{2i-1},w_{2i})=(D,U)$, then $(w_{2i-1}',w_{2i}')=(L_0,L_1)$ and $(\mu_{2i-1}',\mu_{2i}')=(x,y+1)$;
 \item If $(w_{2i-1},w_{2i})=(L_0,U)$, then $(w_{2i-1}',w_{2i}')=(U,L_1)$ and $(\mu_{2i-1}',\mu_{2i}')=(x,y+1)$;
 \item If $(w_{2i-1},w_{2i})=(D,L_0)$, then $(w_{2i-1}',w_{2i}')=(L_0,D)$ and $(\mu_{2i-1}',\mu_{2i}')=(x,y+1)$.
 \end{itemize}
 It is routine to check that $\rho$ sets up an one-to-one correspondence between $M_{2n}$ and $M_{2n}^*$. 
 \end{proof}
 
 By Lemmas~\ref{lem:5},~\ref{lem:6},~\ref{lem:7} and~\ref{lem:8}, the composition $\Phi:=\phi_{FV}^{-1}\circ\rho\circ\phi_{FV}\circ\eta$ is a bijection between $\Ec_{2n}$ and $\D_{2n-1}$. 
 
\section{The construction of $\psi$ and $\Psi$}
 \label{sec:3}

\subsection{ The construction of $\psi$}
It is clear that every cycle $\alpha$ on $M$ can be written uniquely as $\alpha=(a_1^{l_1},a_2^{l_2},\ldots,a_k^{l_k})$, called the {\em compact form} of $\alpha$, where $a_i\neq a_{i+1}$ for $1\leq i\leq k$ (by convention $a_{k+1}=a_1$) and $l_i\geq 1$, that is, all the adjacency letters with the same values are pinched into a bundle. For example, the compact form of the cycle $(1,2,2,1,1,1,3,4,4,2,1)$ is $(1^2,2^2,1^3,3,4^2,2)$. A bundle $a_i^{l_i}$ ($1\leq i\leq k$) is called a {\em cyclic  double ascent} (resp.~ {\em cyclic  double descent}) of $\alpha$ if $a_{i-1}<a_i<a_{i+1}$ (resp.~$a_{i-1}>a_i>a_{i+1}$). The parity of a bundle $a_i^{l_i}$ is the parity of $a_i$. Now if $\alpha\in\Ec_M$, then define $\psi(\alpha)$ to be the cycle  obtained from $\alpha$ by moving each even cyclic double ascent bundle to the place immediately before the closest ({\bf in clockwise direction}) bundle with smaller value. For example (see Fig.~\ref{fig:psi}), if 
$$\alpha=(1,2^2,4^3,6,5^2,6,1^2,8,1^2,4,5,8,3^2,4),$$
then 
$$
\psi(\alpha)=(1,6,5^2,6,4^3,2^2,1^2,8,1^2,5,8,4,3^2,4). 
$$

 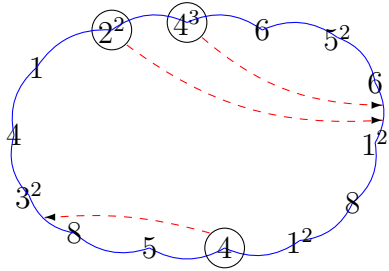
\begin{figure}[h]
\centering
\begin{tikzpicture}[scale=1]
\node at (0,0) {$1$};\node at (1,0.5) {$2^2$};
\node at (1,0.5){\circle{15}};\node at (2,0.6){\circle{15}};
\draw[bend right=20,color=red,dashed] (2.2,0.4) to (4.6,-0.5);\node at (4.6,-0.5){\vector(1,0){0.1}};
\draw[bend right=20,color=red,dashed] (1.2,0.3) to (4.6,-0.7);\node at (4.6,-0.7){\vector(1,0){0.1}};
\draw[bend left=30,color=blue] (0,0) to (1,0.5);\draw[bend left=30,color=blue] (1,0.5) to (2,0.6);
\node at (2,0.6) {$4^3$}; \draw[bend left=30,color=blue] (2,0.6) to (3,0.5);
\node at (3,0.5) {$6$};\draw[bend left=30,color=blue] (3,0.5) to (4,0.4);
\node at (4,0.4) {$5^2$};\draw[bend left=30,color=blue] (4,0.4) to (4.5,-0.2);
\node at (4.5,-0.2) {$6$};\draw[bend left=30,color=blue] (4.5,-0.2) to (4.5,-1);
\node at (4.5,-1) {$1^2$};\draw[bend left=30,color=blue] (4.5,-1) to (4.2,-1.8);
\node at (4.2,-1.8) {$8$};\draw[bend left=30,color=blue] (4.2,-1.8) to (3.5,-2.3);
\node at (3.5,-2.3) {$1^2$};\draw[bend left=30,color=blue]  (3.5,-2.3) to (2.5,-2.4);

\node at (2.5,-2.4) {$4$};\draw[bend left=30,color=blue]  (2.5,-2.4) to (1.5,-2.4);
\node at (2.5,-2.4){\circle{15}};
\draw[bend right=10,color=red,dashed] (2.3,-2.2) to (0.1,-2);\node at (0.1,-1.99){\vector(-1,0){0.1}};
\node at (1.5,-2.4) {$5$};\draw[bend left=30,color=blue]  (1.5,-2.4) to (0.5,-2.2);
\node at (0.5,-2.2) {$8$};\draw[bend left=30,color=blue]  (0.5,-2.2) to (-0.1,-1.7) ;
\node at (-0.1,-1.7) {$3^2$};\draw[bend left=30,color=blue]  (-0.1,-1.7)  to (-0.3,-0.9);
\node at (-0.3,-0.9) {$4$};\draw[bend left=30,color=blue]  (-0.3,-0.9) to (0,0);
\end{tikzpicture}
\caption{An example of $\psi$. \label{fig:psi}}
\end{figure}
 Two key observations about $\psi$ are: 
 \begin{itemize}
 \item the resulting cycle $\psi(\alpha)$ is independent of the order of the movings; 
 \item if the bundle $a_i^{l_i}$ is an even cyclic double ascent of $\alpha$, then $a_i^{l_i}$ becomes an even cyclic double descent bundle of $\psi(\alpha)$.
 \end{itemize}
 Moreover, it is routine to check that $\psi(\alpha)\in\Dc_M$. To see that $\psi$ is a bijection between $\Ec_M$ and $\Dc_M$, we define its inverse explicitly.  For a given cycle $\alpha\in\Dc_M$,  define $\psi^{-1}(\alpha)$ to be the cycle obtained from $\alpha$ by moving each even cyclic double descent bundle to the place immediately before the closest ({\bf in anti-clockwise direction}) bundle with smaller value. It is routine to check that $\psi$ and $\psi^{-1}$ are inverse of each other and thus $\psi$ is indeed  a bijection. 
 
 \subsection{The construction of $\Psi$}
 Suppose that $\alpha=(a_1,a_2,\ldots,a_{2n})\in\Dc_{2n}$ with $a_1=1$ and $a_k=2n$ for some $k$, then define $\vartheta(\alpha)$ to be the permutation 
 $$a_{k+1},a_{k+2},\ldots,a_{2n},a_1,a_2,\ldots,a_{k-1} \text{ (in one line notation)},
 $$
 which is clearly in $\D_{2n-1}$. For example, if $\alpha=(1,5,6,3,4,2)\in\Dc_{6}$, then $\vartheta(\alpha)=34215\in\D_5$. Thus, $\vartheta$ sets up an one-to-one correspondence between $\Dc_{2n}$ and $\D_{2n-1}$. Now define $\Psi$ to be the composition $\vartheta\circ\psi$, which is another  bijection between $\Ec_{2n}$ and $\D_{2n-1}$ in view of Theorem~\ref{thm:mul}. 
 
 \section{An Inclusion-Exclusion approach to Dumont's result~\eqref{eq:dumont}}
 \label{sec:4}
 For the sake of completeness, this section is devoted to an Inclusion-Exclusion approach to Dumont's result~\eqref{eq:dumont}. Our starting point is the following expression for Genocchi numbers deduced by Dumont~\cite[Proposition~1]{du}:
 \begin{equation}\label{exp:dum}
 g_{n+1}=\sum(-1)^{n-u_n}(u_1u_2\cdots u_n)^2,
 \end{equation}
 summed over all $(u_1,u_2,\ldots,u_n)$ such that $u_1=1$ and $u_i$ equals $u_{i-1}$ or $u_{i-1}+1$ for $2\leq i\leq n$. For example, $g_4=(1\cdot1\cdot1)^2-(1\cdot1\cdot2)^2-(1\cdot2\cdot2)^2+(1\cdot2\cdot3)^2=17$. 
 
 For a permutation $\sigma\in\S_n$, a letter $\sigma(i)$ is called a {\em descent top} of $\sigma$ if $i\in[n-1]$ and $\sigma(i)>\sigma(i+1)$. Denote by $\DT(\sigma)$ the set of all descent tops of $\sigma$. For example, $\DT(34215)=\{2,4\}$. For any $S\subseteq[2,n]:=\{2,3,\ldots,n\}$, let us introduce
 \begin{align*}
 \DT_=(S,n)&:=\{\sigma\in\S_n:\DT(\sigma)=S\},\\
 \DT_{\leq}(S,n)&:=\{\sigma\in\S_n:\DT(\sigma)\subseteq S\}.
 \end{align*}
 Let $f_=(S,n)= |\DT_=(S,n)|$ and $f_{\leq}(S,n)= |\DT_{\leq}(S,n)|$.
 Then, it follows from the principle of Inclusion-Exclusion  (see~\cite[Sec.~2.1]{Stan}) that
 \begin{equation}\label{eq:IE}
f_=(S,n)=\sum_{T\subseteq S}(-1)^{|S-T|}f_{\leq}(T,n).
 \end{equation}
 Suppose that $S=\{s_1,s_2,\ldots,s_k\}\subseteq[2,n]$ with $s_1>s_2>\cdots>s_k>1$. Let $d_i(S)=s_i-s_{i+1}$ for $i\in[k-1]$ and $d_k(S)=s_k-1$. We have the following product formula for $f_{\leq}(T,n)$.
 \begin{lemma}
 Let $T\subseteq S$ and let $u_i=u_i(T):=1+|\{t\in T: t\geq s_i\}|$. Then
 \begin{equation}\label{eq:leq}
 f_{\leq}(T,n)=\prod_{i=1}^k u_i^{d_i(S)}. 
 \end{equation} 
 \end{lemma}
 \begin{proof}
For any letter $\ell$, $s_{i+1}\leq \ell< s_{i}$, and any partial permutation $p$ of $\{n,n-1,\ldots,\ell+1\}$ whose descent top includes in $T$, there are exactly $u_i$ positions to insert the letter $\ell$ into $p$ to obtain a partial permutation of $\{n,n-1,\ldots,\ell\}$ with descent top includes in $T$. These $u_i$ positions are the leftmost space of $p$ plus the spaces immediately after each letter from $\{t\in T: t\geq s_i\}$. The desired product formula for $f_{\leq}(T)$ then follows. 
\end{proof}

Combining~\eqref{eq:IE} and~\eqref{eq:leq} we have the following formula for $f_=(S,n)$ that was obtained by Chang, Ma and Yeh~\cite[Theorem~1.1]{Ma} via different approach.

\begin{theorem}[Chang, Ma and Yeh~\cite{Ma}]\label{thm:Ma}
For any $S\subseteq[2,n]$ with $|S|=k$, we have 
\begin{equation}
f_=(S,n)=\sum (-1)^{k+1-u_{k}}\prod_{i=1}^k u_i^{d_i(S)},
\end{equation}
 summed over all $(u_0,u_1,u_2,\ldots,u_k)$ such that $u_0=1$ and $u_i$ equals $u_{i-1}$ or $u_{i-1}+1$ for $1\leq i\leq k$. 
\end{theorem}

Since $\D_{2n+1}=\DT_=(\{2,4,\ldots,2n\},2n+1)$, it follows from Theorem~\ref{thm:Ma} that 
$$
|\D_{2n+1}|=\sum(-1)^{n+1-u_{n}}(u_1u_2\cdots u_{n-1})^2u_n
$$
 summed over all $(u_0,u_1,u_2,\ldots,u_n)$ such that $u_0=1$ and $u_i$ equals $u_{i-1}$ or $u_{i-1}+1$ for $1\leq i\leq n$. As $u_n=u_{n-1}$ or $u_n=u_{n-1}+1$, the above summation is simplified to the right-hand side of~\eqref{exp:dum}, which proves  $ |\D_{2n+1}|=g_{n+1}$.
 
\section*{Acknowledgments}
 This work was supported by the National Science Foundation of China grants  11871247 and 12071440, and the project of Qilu Young Scholars of Shandong University.

 \end{document}